\documentclass[11pt]{article}
\usepackage{amsmath}
\usepackage{amssymb}

\newtheorem{lemma}{Lemma}[section]
\newtheorem{proposition}{Proposition}[section]
\newtheorem{theorem}{Theorem}[section]
\newtheorem{corollary}{Corollary}[section]
\newtheorem{example}{Example}[section]
\numberwithin{equation}{section}

\newcommand{\Ric}{\mathop{\rm Ric}}
\newcommand{\id}{\mathop{\rm id}}
\newcommand{\spanop}{\mathop{\rm span}}

\begin{document}

\title{Geometric Quantities in Spray Geometry and Their Applications to Finsler Geometry}
\author{}
\date{}
\maketitle

\begin{center}
{\it
This is the English translation of the paper in SCIENTIA SINICA Mathematica:  B. Li and Z. Shen,  Geometric quantities in spray geometry and their applications in Finsler geometry (in Chinese).
Sci Sin Math, 2024, 54: 116, doi: 10.1360/SSM-2023-0106
\vskip 2mm
This English translation is for readers who do not read Chinese.
It is not for separate publication or journal submission.
The original Chinese article prevails; copyright remains with the original publisher/authors.
}
\end{center}

\begin{center}
Dedicated to Professor Yibing Shen on the occasion of his 85th birthday
\end{center}

\begin{center}
Benling Li\\
School of Mathematics and Statistics, Ningbo University,\\
Ningbo 315211, China\\
E-mail: libenling@nbu.edu.cn

\vspace{0.15in}
Zhongmin Shen\\
Department of Mathematics, Indiana University,\\
Indianapolis, Indiana 46202-3216, U.S.A.\\
E-mail: zshen@iu.edu
\end{center}

\begin{abstract}
In this paper we give an overview of several developments in spray geometry and provide some further interpretations of the geometric quantities that arise naturally in this subject. We discuss, in particular, the $S$-curvature, the $\chi$-curvature, and several notions of Ricci curvature for sprays. We rederive a number of identities relating these quantities and explain their applications to locally projectively flat sprays and to projectively related sprays. We also discuss applications to Finsler geometry, with emphasis on the $\chi$-curvature of a Finsler metric and on the inverse problem of determining whether a given spray is induced by a Finsler metric.


\end{abstract}

\section{Introduction}

Consider a manifold $M$ equipped with a distinguished family of parametrized curves, called geodesics. Spray geometry is the study of the geometric properties determined by such a family of curves. The geodesics are described by a special vector field $G$ on the tangent bundle $TM$, called a spray.

Every (pseudo-)Finsler metric on a manifold induces a spray. Consequently, many questions in Finsler geometry can be investigated through the spray associated with the metric. One of the main themes of this article is to study Finsler geometric problems by means of the corresponding spray. At the same time, spray geometry is broader than Finsler geometry: in general, a spray need not be induced by a positive definite Finsler metric, or even by a Finsler metric at all. Thus sprays form a natural geometric category in their own right.

Two fundamental curvature quantities of a spray are its Riemann curvature and its Berwald curvature. The Riemann curvature of a spray is a natural extension of the Riemann curvature of a Riemannian metric. The Berwald curvature, on the other hand, is genuinely non-Riemannian: it vanishes whenever the spray is induced by a Riemannian metric. These two curvature tensors and their traces, namely the Ricci curvature and the mean Berwald curvature, play central roles in spray geometry.

Let $G$ be a spray on $M$. At each nonzero tangent vector $y\in T_xM$, the Riemann curvature may be viewed as a linear map
\[
R_y:T_xM\longrightarrow T_xM,
\qquad R_y(y)=0.
\]
As a linear transformation, $R_y$ need not be diagonalizable. A spray is said to have scalar curvature if
\[
R_y=R(y)\id-\tau_y\otimes y,
\qquad \tau_y(y)=R(y).
\]
For a spray of scalar curvature, $R_y$ is always diagonalizable. It is an interesting problem to understand the global geometry of closed manifolds carrying sprays of scalar curvature.

Spray geometry is generally nonreversible. If $c=c(t)$ is a geodesic, the reversed curve $\bar c(t)=c(t_0-t)$ need not be a geodesic. A spray is called forward complete if every geodesic $c:(a,b)\to M$ can be extended to a geodesic defined on $(a,+\infty)$. A spray on a closed manifold need not be forward complete. If the spray is forward complete, the exponential map
\[
\exp_p:T_pM\longrightarrow M
\]
is defined on the whole tangent space. Unlike the Riemannian case, however, its image need not be the whole manifold. In the absence of a metric there is, in general, no intrinsic notion of the length of a geodesic. The image of $T_pM$ under $\exp_p$ may be regarded as an observable domain. This raises a natural global question: for a closed manifold $M$, what is the smallest number of observable domains needed to cover $M$?

If $M$ is an $n$-dimensional manifold and $dV$ is a nowhere vanishing $n$-form, then a pair $(G,dV)$ determines an important quantity called the $S$-curvature. From the $S$-curvature one obtains the one-form
\[
\chi=\chi_kdx^k,
\qquad
\chi_k={1\over2}\bigl(S_{\cdot k|m}y^m-S_{|k}\bigr).
\]
A remarkable fact is that $\chi$ is independent of the choice of volume form $dV$ and therefore is an intrinsic quantity of the spray. We shall give several expressions for $\chi$ in terms of the Riemann curvature. Different expressions reveal different geometric meanings. In particular, when the spray is induced by a Finsler metric, the $\chi$-curvature can be expressed in terms of the mean Landsberg curvature. This leads to a rigidity theorem: a complete Finsler metric of negative flag curvature, with vanishing $\chi$-curvature and bounded mean Cartan torsion, must be Riemannian.

The $S$-curvature also provides a natural way to deform a spray within its projective class. Given a volume form, one can deform $G$ to a projectively related spray $\widehat G$ whose $S$-curvature vanishes. Any curvature quantity of $\widehat G$ is then a projective invariant once the volume form is fixed. If a geometric quantity of $\widehat G$ turns out to be independent of the volume form, it becomes a genuine projective invariant of $G$. The Douglas curvature and Weyl curvature arise in this way. The same construction also leads naturally to projective versions of the Riemann and Ricci curvatures.

Locally projectively flat sprays form an important class. In dimension at least three they are characterized by the simultaneous vanishing of the Douglas curvature and the Weyl curvature. Their local structure is relatively simple, but their global geometry and topology remain much less understood. One striking result is that a manifold carrying a locally projectively flat spray has vanishing real Pontrjagin classes. This suggests that the study of forward complete sprays subject to curvature restrictions should reveal interesting interactions between local spray geometry and global topology.

In the final part of this article we discuss the inverse problem of spray geometry. Every Finsler metric induces a spray, but the converse is false. Thus one would like to characterize those sprays that are metrizable by Finsler metrics. We review several results for sprays of scalar or isotropic curvature and, in particular, for locally projectively flat sprays.

\section{Sprays}

A \emph{spray} on a manifold $M$ is a vector field $G$ on $TM$ satisfying the following conditions:

\noindent (i) $G$ is $C^\infty$ on $TM\setminus\{0\}$;

\noindent (ii) in local coordinates $(x^i,y^i)$ on $TM$,
\[
G=y^i{\partial\over\partial x^i}-2G^i{\partial\over\partial y^i},
\]
where each $G^i=G^i(x,y)$ is positively homogeneous of degree two in $y$:
\[
G^i(x,\lambda y)=\lambda^2G^i(x,y),\qquad \lambda>0.
\]

The geodesics of $G$ are the projections of the integral curves of $G$. Thus, in local coordinates, a curve $c(t)=(x^i(t))$ is a geodesic if and only if
\[
{d^2x^i\over dt^2}+2G^i\left(x,{dx\over dt}\right)=0.
\]

Every Finsler metric $L$ on $M$ induces a spray whose coefficients are
\[
G^i={1\over4}g^{il}\{L_{x^my^l}y^m-L_{x^l}\},
\]
where $(g^{il})=(g_{il})^{-1}$ and
\[
g_{ij}={1\over2}L_{y^iy^j}.
\]
Here a Finsler metric $L$ is not necessarily assumed to be positive definite. It satisfies:

\noindent (1) $L$ is $C^\infty$ on $TM\setminus\{0\}$;

\noindent (2) $L(x,\lambda y)=\lambda^2L(x,y)$ for $\lambda>0$;

\noindent (3) the matrix
\[
g_{ij}(x,y)={1\over2}{\partial^2L\over\partial y^i\partial y^j}(x,y)
\]
is nonsingular.

If $g_{ij}$ is positive definite, then $L$ is called positive definite. In this case one usually writes $F=\sqrt L$.

A spray is called \emph{forward complete} if every geodesic $c:(a,b)\to M$ extends to a geodesic defined on $(a,+\infty)$.

\begin{example}[\cite{Sh1}]
On ${\bf R}^2$, consider
\[
G=y^1{\partial\over\partial x^1}+y^2{\partial\over\partial x^2}
-2G^1{\partial\over\partial y^1}-2G^2{\partial\over\partial y^2},
\]
where
\[
G^1={y^2\over2r}\sqrt{(y^1)^2+(y^2)^2},
\qquad
G^2=-{y^1\over2r}\sqrt{(y^1)^2+(y^2)^2},
\]
and $r>0$ is constant. Its geodesic equations are
\[
{d^2x^1\over dt^2}+{1\over2r}
\sqrt{\left({dx^1\over dt}\right)^2+\left({dx^2\over dt}\right)^2}
{dx^2\over dt}=0,
\]
\[
{d^2x^2\over dt^2}-{1\over2r}
\sqrt{\left({dx^1\over dt}\right)^2+\left({dx^2\over dt}\right)^2}
{dx^1\over dt}=0.
\]
Their general solution can be written as
\[
x^1=r\cos\left({\lambda\over r}t+\theta\right)-r\cos\theta+x^1(0),
\]
\[
x^2=r\sin\left({\lambda\over r}t+\theta\right)-r\sin\theta+x^2(0).
\]
Hence $G$ is forward complete.
\end{example}

\begin{example}
Let
\[
G=y^i{\partial\over\partial x^i}-2G^i{\partial\over\partial y^i}
\]
be a spray on ${\bf R}^n$ with
\[
G^i=-{x^i\over\sum_{j=1}^n(x^j)^2}(y^i)^2.
\]
This spray is not forward complete.

Let $\lambda>0$ be a constant and consider
\[
\phi:{\bf R}^n\setminus\{0\}\longrightarrow {\bf R}^n\setminus\{0\},
\qquad x\longmapsto\lambda x.
\]
The map $\phi$ preserves geodesics. Let
\[
\Gamma=\{\id,\phi^m:m\in{\bf Z}\}
\]
act on ${\bf R}^n\setminus\{0\}$. Then
\[
M=({\bf R}^n\setminus\{0\})/\Gamma
\]
is closed, and the spray induced by $G$ on $M$ is not forward complete.
\end{example}

\section{Riemann Curvature and Berwald Curvature}

Let
\[
G=y^i{\partial\over\partial x^i}-2G^i{\partial\over\partial y^i}
\]
be a spray on $M$, and let
\[
VTM=\spanop\left\{{\partial\over\partial y^i}\right\}.
\]
The spray determines a horizontal distribution
\[
HTM=\spanop\left\{{\delta\over\delta x^i}\right\},
\qquad
{\delta\over\delta x^i}={\partial\over\partial x^i}-N^j_i{\partial\over\partial y^j},
\]
where
\[
N^i_j={\partial G^i\over\partial y^j}.
\]
On $TM\setminus\{0\}$ we have the vector fields
\[
Y=y^i{\partial\over\partial y^i},
\qquad
G=y^i{\delta\over\delta x^i}.
\]
Let $\{dx^i,\delta y^i\}$ denote the coframe dual to
\[
\left\{{\delta\over\delta x^i},{\partial\over\partial y^i}\right\},
\]
so that
\[
\delta y^i=dy^i+N^i_jdx^j.
\]
Set
\[
\omega^i=dx^i,\qquad \omega^{n+i}=\delta y^i,
\qquad \omega^i_j=\Gamma^i_{jk}dx^k,
\]
where
\[
\Gamma^i_{jk}={\partial^2G^i\over\partial y^j\partial y^k}.
\]
Also set
\[
e_i={\delta\over\delta x^i},
\qquad e_{n+i}={\partial\over\partial y^i}.
\]

\begin{lemma}[\cite{Sh1}]
The structure equations satisfy
\[
d\omega^i=\omega^j\wedge\omega^i_j.
\]
\end{lemma}

Let
\[
\Omega^{\ i}_j=d\omega^i_j-\omega^k_j\wedge\omega^i_k.
\]
The curvature tensor can be written as
\[
\Omega=\Omega^{\ i}_j\,\omega^j\otimes e_i.
\]
Write
\[
\Omega^{\ i}_j={1\over2}R^{\ i}_{j\ kl}\omega^k\wedge\omega^l
-{1\over2}B^{\ i}_{j\ kl}\omega^k\wedge\omega^{n+l},
\]
where $R^{\ i}_{j\ kl}+R^{\ i}_{j\ lk}=0$. Then
\[
B^{\ i}_{j\ kl}={\partial\Gamma^i_{jk}\over\partial y^l}
={\partial^3G^i\over\partial y^j\partial y^k\partial y^l}.
\]
The Riemann curvature coefficients
\[
R^i_{\ k}=y^jR^{\ i}_{j\ kl}y^l
\]
are given by
\[
R^i_{\ k}=2{\partial G^i\over\partial x^k}
-y^m{\partial^2G^i\over\partial y^k\partial x^m}
+2G^m{\partial^2G^i\over\partial y^m\partial y^k}
-{\partial G^m\over\partial y^k}{\partial G^i\over\partial y^m}.
\]

If $B=0$, then $G$ is called a \emph{Berwald spray}, or equivalently an affine spray. In this case
\[
G^i={1\over2}\Gamma^i_{jk}(x)y^jy^k
\]
is quadratic in $y$ at each $x\in M$. For an affine spray, $R^{\ i}_{j\ kl}=R^{\ i}_{j\ kl}(x)$ is independent of $y$.

\begin{lemma}[\cite{Sh1}]
\begin{align}\label{R_ijkl}
R^{\ i}_{j\ kl}={1\over3}
\left\{{\partial^2R^i_{\ k}\over\partial y^j\partial y^l}
-{\partial^2R^i_{\ l}\over\partial y^j\partial y^k}\right\}.
\end{align}
\end{lemma}

The Riemann curvature $R^i_{\ k}\omega^k\otimes e_i$ may be viewed as a family of linear maps
\[
R_y:T_xM\longrightarrow T_xM,
\]
where
\[
R_y\left({\partial\over\partial x^k}\Big|_x\right)
=R^i_{\ k}(x,y){\partial\over\partial x^i}\Big|_x.
\]

\section{$S$-Curvature and $\chi$-Curvature}

Let
\[
G=y^i{\partial\over\partial x^i}-2G^i{\partial\over\partial y^i}
\]
be a spray on an $n$-dimensional manifold $M$, and let
\[
dv=\sigma(x)dx^1\cdots dx^n
\]
be a volume form. The $S$-curvature of $(G,dv)$ is
\[
S={\partial G^m\over\partial y^m}-y^m{\partial\over\partial x^m}\ln\sigma.
\]
This quantity plays an important role in spray geometry.

Let $\phi_t:TM\to TM$ denote the geodesic flow of $G$. If $d\hat{v}$ is the induced volume form on $TM$,
\begin{align*}
\begin{split}
 d\hat{v} & = \sigma^2(x) dx^1 \cdot\cdot\cdot dx^n dy^1 \cdot\cdot\cdot dy^n \\
          & = \sigma^2(x) dx^1 \cdot\cdot\cdot dx^n \delta y^1 \cdot\cdot\cdot \delta y^n.
\end{split}
\end{align*}
then
\[
{d\over dt}\left[(\phi_t)^*d\hat{v}\right]_{t=0}
=-2Sd\hat{v}.
\]
Thus, if $S=0$, then $(\phi_t)^* d\hat{v}= d\hat{v}.$

The $\chi$-curvature $\chi=\chi_kdx^k$ is defined by
\begin{align}\label{chi}
\chi_k={1\over2}\left(S_{\cdot k|m}y^m-S_{|k}\right).
\end{align}
In local coordinates,
\begin{align}\label{chi_local}
\chi_k={1\over2}\left[
 y^m{\partial^2\Pi\over\partial x^m\partial y^k}
-2G^m{\partial^2\Pi\over\partial y^m\partial y^k}
-{\partial\Pi\over\partial x^k}\right],
\end{align}
where
$
\Pi={\partial G^m\over\partial y^m}.
$
Hence $\chi_k$ is independent of the volume form $dv$.

\section{Identities Involving the $\chi$-Curvature}

In this section we discuss several expressions for the $\chi$-curvature. This quantity appears naturally in a number of important geometric identities.

The mean Berwald curvature is defined by
\[
E_{ij}={1\over2}B^{\ m}_{i\ mj}.
\]
Since
\[
E_{ij}={1\over2}S_{\cdot i\cdot j},
\]
we define the $H$-curvature by
\[
H_{ij}=E_{ij|m}y^m.
\]
Thus $H=H_{ij}dx^i\otimes dx^j$ is the covariant derivative of $E$ along geodesics.

\begin{lemma}[\cite{Sh2}]
The $H$-curvature is determined by the $\chi$-curvature according to
\begin{align}\label{H_ij}
H_{ij}={1\over2}(\chi_{i\cdot j}+\chi_{j\cdot i}).
\end{align}
\end{lemma}

\noindent {\bf Proof.}
By the Ricci identity,
\[
S_{\cdot i|j}=S_{|j\cdot i}.
\]
Therefore
\begin{align*}
\begin{split}
\chi_{i \cdot j} & = \frac{1}{2}[ S_{\cdot i |m \cdot j} y^m + S_{ \cdot i| j} -S_{| i \cdot j} ]  = \frac{1}{2}[ S_{\cdot i  \cdot j |m} y^m + S_{ \cdot i| j} -S_{| i \cdot j} ] = E_{ij|m}y^m + \frac{1}{2}(S_{ \cdot i| j} -S_{| i \cdot j}),
\end{split}
\end{align*}
\begin{align*}
\begin{split}
\chi_{i \cdot j} + \chi_{j \cdot i}& = 2E_{ij|m}y^m +  \frac{1}{2}(S_{ \cdot i| j} -S_{| i \cdot j}) = 2E_{ij|m}y^m.
\end{split}
\end{align*}
Adding the corresponding formula with $i$ and $j$ interchanged yields (\ref{H_ij}). Contracting (\ref{H_ij}) with $y^i$, we obtain
\[
H_{ij}y^i  = \frac{1}{2}( \chi_{i \cdot j} y^i + \chi_{ j\cdot i} y^i )  = \frac{1}{2}[( \chi_i y^i )_{\cdot j} - \chi_j]  = -\frac{1}{2}\chi_j.
\]
and hence
\[
\chi_j=-2H_{ij}y^i=-2E_{ij|m}y^my^i
=-S_{\cdot i\cdot j|m}y^my^i.
\]
\hfill {\it Q.E.D.}

\begin{lemma}[\cite{Sh4}]\label{chi_k_R}
The $\chi$-curvature can be expressed in terms of the Riemann curvature by
\[
\chi_k=-{1\over2}R^{\ m}_{m\ kl}y^l.
\]
\end{lemma}

\noindent {\bf Proof.}
From (\ref{R_ijkl}),
\[
R^{\ i}_{j\ kl}y^l={1\over3}(2R^i_{\ k\cdot j}+R^i_{\ j\cdot k}).
\]
Hence
\[
\chi_k=-{1\over6}(2R^m_{\ k\cdot m}+R^m_{\ m\cdot k})
=-{1\over2}R^{\ m}_{m\ kl}y^l.
\]
\hfill {\it Q.E.D.}

Define
\[
T^i_{\ k}:=R^i_{\ k}-\left(R\delta^i_k-{1\over2}R_{\cdot k}y^i\right),
\qquad
R={1\over n-1}R^m_{\ m}.
\]

\begin{lemma}[\cite{Sh4}]
The $\chi$-curvature is given by
\[
\chi_k=-{1\over3}T^m_{\ k\cdot m}.
\]
\end{lemma}

\noindent {\bf Proof.}
By definition,
\[
\begin{array}{rcl}
T^m_{\ k\cdot m}
&=&R^m_{\ k\cdot m}
-\left(R_{\cdot m}\delta^m_k
-{1\over2}R_{\cdot k\cdot m}y^m
-{1\over2}R_{\cdot k}\delta^m_m\right)\\[3pt]
&=&R^m_{\ k\cdot m}+{n-1\over2}R_{\cdot k}\\[3pt]
&=&{1\over2}(2R^m_{\ k\cdot m}+R^m_{\ m\cdot k})\\[3pt]
&=&-3\chi_k.
\end{array}
\]
\hfill {\it Q.E.D.}

In particular, if $T^i_{\ k}=0$, namely if $G$ has isotropic curvature, then $\chi_k=0$.
The relationship between the Weyl tensor, the $\chi$-curvature, and the tensor $T$ is given by the following lemma.
\begin{lemma}[\cite{Sh4}]\label{lemWeyl}
The Weyl curvature is given by
\begin{align}\label{W_ik}
W^i_{\ k}=T^i_{\ k}+{3\over n+1}\chi_k y^i.
\end{align}
\end{lemma}

\noindent {\bf Proof.}
Using the standard formula
\[
W^i_{\ k}=A^i_{\ k}-{1\over n+1}A^m_{\ k\cdot m}y^i,
\qquad
A^i_{\ k}=R^i_{\ k}-R\delta^i_k,
\]
we obtain
\[
W^i_{\ k}
=R^i_{\ k}-\left(R\delta^i_k-{1\over2}R_{\cdot k}y^i\right)
-{1\over2(n+1)}(2R^m_{\ k\cdot m}+R^m_{\ m\cdot k})y^i = T^i_{\ k} + \frac{3}{n+1} \chi_k y^i,
\]
which is precisely (\ref{W_ik}).
\hfill {\it Q.E.D.}

The condition $W^i_{\ k}=0$ is equivalent to the condition that $G$ has scalar curvature; see Section 9.

\section{Ricci Curvature and Ricci Curvature Tensors}

It is natural to consider Ricci curvature for a spray. However, there are several possible ways to take traces of the Riemann curvature.

Recall that the Riemann curvature may be viewed as a family of linear maps $R_y:T_xM\to T_xM$. The trace of $R_y$ is called the Ricci curvature:
\[
\Ric=R^m_{\ m}.
\]
Thus $\Ric$ is a scalar function on $TM$, rather than a tensor on $TM\setminus\{0\}$.

On the other hand, from (\ref{R_ijkl}), the Riemann curvature tensor $R^{\ i}_{j\ kl}$ of the Berwald connection is closely related to $R^i_{\ k}$. One may take its trace by setting
\[
{\Ric}_{jl}:=R^{\ m}_{j\ ml}.
\]
Then
\[
{\Ric}={\Ric}_{jl}y^jy^l.
\]
Using (\ref{R_ijkl}), one finds
\begin{align*}
\begin{split}
{\Ric}_{jl} & =  \frac{1}{3}\Big \{  R^m_{\ m \cdot l \cdot j} - R^m_{\ l \cdot m \cdot j}\Big \}\nonumber\\
& =  \frac{1}{3}\Big \{  R^m_{\ m \cdot l \cdot j} - R^m_{\ l \cdot m \cdot j}\Big \}\nonumber\\
& =  \frac{1}{2} R^m_{\ m\cdot j\cdot l}- \frac{1}{6} R^m_{\ m\cdot j \cdot l} -\frac{1}{3} R^m_{\ l\cdot m \cdot j} \nonumber\\
& =  \frac{1}{2} R^m_{\ m\cdot j\cdot l}- \frac{1}{6} \Big \{ 2 R^m_{\ l\cdot j} + R^m_{\ m \cdot l} \Big \}_{\cdot j}\nonumber\\
& =  \frac{1}{2} {\Ric}_{\cdot j \cdot l}+ \chi_{l\cdot j}.
\end{split}
\end{align*}
Let
\[
\widetilde{\Ric}_{jl}:={1\over2} {\Ric}_{\cdot j\cdot l}.
\]
We therefore obtain the following.

\begin{lemma}[\cite{SL1}]
The Ricci curvature tensor and the Ricci curvature are related by
\begin{align}\label{Ric_jl}
{\Ric}_{jl}={1\over2} {\Ric}_{\cdot j\cdot l}+\chi_{l\cdot j}.
\end{align}
\end{lemma}

Contracting (\ref{Ric_jl}) with $y^j$ gives
\[
\chi_l=({\Ric}_{jl}-\widetilde{\Ric}_{jl})y^j.
\]

A spray is called \emph{Ricci flat} if $\Ric=0$. It is called \emph{strongly Ricci flat} if, in addition, ${\Ric}_{jl}=0$.
By symmetrizing ${\Ric}_{jl}$, define
\[
\overline{\Ric}_{jl}:={1\over2}({\Ric}_{jl}+{\Ric}_{lj}).
\]
Then
\[
\overline{\Ric}_{jl}=\widetilde{\Ric}_{jl}+H_{jl}.
\]
It follows that
\[
\overline{\Ric}_{jl}=0
\quad\Longleftrightarrow\quad
{\Ric}=0\quad\hbox{and}\quad H_{jl}=0.
\]
Consequently:
\begin{proposition}[\cite{SL1}]
$\Ric_{jl}=0$ is equivalent to
\[
\Ric=0
\qquad\hbox{and}\qquad
\chi_k=0.
\]
\end{proposition}

For a general spray one may consider curvature conditions such as $\Ric>0$, $\Ric<0$, $\overline{\Ric}_{ij}>0$, or $\overline{\Ric}_{ij}<0$. At present, the full geometric significance of these conditions is not yet clear.

If the spray is induced by a positive definite Finsler metric $F$, then one may consider the Einstein condition, namely that the Ricci curvature (tensor) is isotropic,
\[
\Ric=(n-1)c(x)F^2,
\qquad
{\Ric}_{jl}=(n-1)c(x)g_{jl}.
\]
Without a metric structure, there is in general no further natural averaging of the Ricci curvature.

\section{The $S$-Curvature Deformation of a Spray}
Let $G$ be a spray on an $n$-dimensional manifold $M$, and let $dv$ be a volume form on $M$. Let $S=S(G,dv)$ be the $S$-curvature of $(G,dv)$. As observed in Section 4, when $S=0$, the geodesic flow preserves the induced volume form $d\hat{v}$. We shall prove that, for a fixed volume form, every spray is projectively equivalent to a spray with vanishing $S$-curvature. In fact, one can define a new spray $\widehat{G}$ by
\[
\widehat G^i=G^i-{S\over n+1}y^i.
\]
Clearly, $\widehat G$ is projectively related to $G$. Its $S$-curvature is
\[
\begin{array}{rcl}
\widehat S
&=&{\partial\widehat G^m\over\partial y^m}
-y^m{\partial\over\partial x^m}\ln\sigma\\[3pt]
&=&S-{1\over n+1}{\partial\over\partial y^m}(Sy^m)=0.
\end{array}
\]
Therefore
\[
\widehat\chi_k={1\over2}
(\widehat S_{|m\cdot k}y^m-\widehat S_{|k})=0.
\]
We call $\widehat G$ the spray associated with the pair $(G,dv)$.

\begin{lemma}[\cite{Sh4}]
Suppose that $G_1$ and $G_2$ are projectively related sprays on an $n$-dimensional manifold $M$. For any fixed volume form $dv$, the sprays $\widehat G_1$ and $\widehat G_2$ associated with $(G_1,dv)$ and $(G_2,dv)$ coincide.
\end{lemma}

\noindent {\bf Proof.}
Suppose
\[
G^i_1=G^i_2+Py^i.
\]
Then
\[
S_1=S_2+(n+1)P.
\]
Therefore
\[
\widehat G^i_1
= G^i_1 -{S_1\over n+1}y^i = G^i_2 + P y^i - \frac{S_2 + (n+1)P}{n+1} y^i
=G^i_2-{S_2\over n+1}y^i
=\widehat G^i_2.
\]
\hfill {\it Q.E.D.}

Thus, for a fixed volume form, every geometric quantity of $\widehat G$ is a projective invariant of $G$. If such a quantity is independent of the volume form, then it is an intrinsic projective invariant.

\begin{lemma}[\cite{Sh4}]
The Riemann curvature of $\widehat G$ is
\begin{align}\label{Rik}
\widehat R^i_{\ k}
=R^i_{\ k}+\tau\delta^i_k-{1\over2}\tau_{\cdot k}y^i
+{3\chi_k\over n+1}y^i,
\end{align}
where
\[
\tau:=\left({S\over n+1}\right)^2
+{1\over n+1}S_{|m}y^m.
\]
\end{lemma}

\noindent {\bf Proof.}
This follows by direct computation.
\hfill {\it Q.E.D.}

Note that $\widehat R^i_{\ k}$ is a projective invariant for a fixed volume form. However, $\widehat R^i_{\ k}$ depends on $dv$. From (\ref{Rik}), we have
\[
\widehat R=R+\tau.
\]
Hence
\begin{align*}
\widehat T^i_{\ k}
&=\widehat R^i_{\ k}
-\left(\widehat R\delta^i_k-{1\over2}\widehat R_{\cdot k}y^i\right)\\
&= R^i_{\ k} + \tau \delta^i_k - \frac{1}{2}\tau_{\cdot k} y^i + \frac{3 \chi_k}{n+1} y^i -(R +\tau)\delta^i_k+\frac{1}{2}(R_{\cdot k} +\tau_{\cdot k})y^i\\
&= R^i_{\ k} - (R \delta^i_k -\frac{1}{2} R_{\cdot k} y^i ) + \frac{3 \chi_k}{n+1} y^i \\
&=T^i_{\ k}+{3\chi_k\over n+1}y^i.
\end{align*}
By (\ref{W_ik}), this is precisely the Weyl curvature $W^i_{\ k}$ and is independent of $dv$.

Similarly, the Berwald curvature of $\widehat G$ is independent of $dv$.

\begin{proposition}[\cite{Sh4}]
Both $\widehat B^{\ i}_{j\ kl}$ and $\widehat T^i_{\ k}$ are independent of the volume form $dv$. In fact,
\[
\widehat B^{\ i}_{j\ kl}=D^{\ i}_{j\ kl},
\qquad
\widehat T^i_{\ k}=W^i_{\ k},
\]
where $D$ and $W$ are the Douglas and Weyl curvatures of $G$, respectively.
\end{proposition}

\section{Pontrjagin Classes of Locally Projectively Flat Sprays}

A spray $G$ on an $n$-dimensional manifold $M$ is called \emph{locally projectively flat} if, at every point, there is a local coordinate system in which
\[
G^i=Py^i.
\]

\begin{example}
Let $G_\kappa$ be the spray induced by an $n$-dimensional Riemannian metric $\alpha_\kappa$ of constant sectional curvature $\kappa$. Let $P$ be any positively homogeneous function of degree one on $TM$. Then a projective deformation of $G_\kappa$ by $P$ gives a locally projectively flat spray. In particular, for a given volume form $dv$, one may take
\[
P=-{S\over n+1},
\]
where $S$ is the $S$-curvature of $(G_\kappa,dv)$.
\end{example}

It is natural to ask whether every manifold admits a forward complete locally projectively flat spray, and whether every closed manifold admits a locally projectively flat spray.

Let $G$ and $dv$ be a spray and a volume form on $M$, and let $\widehat G$ be the spray associated with $(G,dv)$:
\[
\widehat G^i=G^i-{S\over n+1}y^i.
\]
The Berwald curvature of $\widehat G$ is the Douglas curvature of $G$:
\[
\widehat B^{\ i}_{j\ kl}=D^{\ i}_{j\ kl}.
\]
Suppose $D^{\ i}_{j\ kl}=0$. Then $\widehat G$ is an affine spray, and its Berwald connection reduces to a torsion-free linear connection on $M$:
\[
\widehat G^i={1\over2}\widehat\Gamma^i_{jk}(x)y^jy^k,
\qquad
\widehat\Gamma^i_{jk}=\widehat\Gamma^i_{kj}.
\]
Its curvature form is a local form on $M$,
\[
\widehat\Omega^{\ i}_j={1\over2}\widehat R^{\ i}_{j\ kl}(x)dx^k\wedge dx^l.
\]
The real Pontrjagin classes of $M$ can be represented in the following form.
\begin{align}\label{P_k}
P_k = \frac{1}{(2k)^{2k}}\Big[ \sigma_{2k}( \widehat{\Omega}) \Big],
\end{align}
where
\begin{align*}
\begin{split}
\sigma_{2k}(\widehat{\Omega}) & = \sum_{i_1 < i_2< \cdot \cdot \cdot < i_{2k}, \ \sigma\in S_{2k}}\operatorname{sgn}(\sigma)
\widehat{\Omega}^{\ i_{\sigma(1)}}_{i_1}\wedge \cdot \cdot \cdot \wedge \widehat{\Omega}^{\ i_{\sigma(2k)}}_{i_{2k}} \\
& = \frac{1}{2^{2k}} \sum_{i_1 < i_2< \cdot \cdot \cdot < i_{2k}, \ \sigma\in S_{2k}}\operatorname{sgn}(\sigma)
\widehat{R}^{\ i_{\sigma(1)}}_{i_1 \ m_1 l_1} \cdot \cdot \cdot \widehat{R}^{\ i_{\sigma(2k)}}_{i_{2k} \ m_{2k} l_{2k}}
dx^{m_1}\wedge dx^{l_1}\wedge \cdot \cdot \cdot \wedge dx^{m_{2k}}\wedge dx^{l_{2k}}.
\end{split}
\end{align*}
It is a $4k$-form on $M$.

\begin{theorem}[\cite{SZ}]
If a manifold $M$ admits a locally projectively flat spray, then all real Pontrjagin classes of $M$ vanish.
\end{theorem}

\noindent {\bf Proof.}
Let $dv$ be a volume form, and let $\widehat G$ be the spray associated with $(G,dv)$.
Since the Douglas curvature of $G$ vanishes, the Pontrjagin classes $P_k$ of $M$ can be represented by the closed forms (\ref{P_k}) on $M$.

The Weyl curvature of $G$ is
\[
W^i_{\ k}=\widehat T^i_{\ k}
=\widehat R^i_{\ k}-\left(\widehat R\delta^i_k-\frac{1}{2}\widehat R_{\cdot k}y^i\right).
\]
Since $W^i_{\ k}=0$, we have
\[
\widehat R^i_{\ k}
=\widehat R\delta^i_k-\frac{1}{2}\widehat R_{\cdot k}y^i.
\]
Then
\begin{align*}
\begin{split}
\widehat R^{\ i}_{j\ kl}
&=\frac{1}{2}\left(\widehat R_{\cdot l\cdot j}\delta^i_k
-\widehat R_{\cdot k\cdot j}\delta^i_l\right),\\
\widehat\Omega^{\ i}_j
&=\frac{1}{4}\left(\widehat R_{\cdot l\cdot j}\delta^i_k
-\widehat R_{\cdot k\cdot j}\delta^i_l\right)dx^k\wedge dx^l\\
&=\frac{1}{2}\widehat R_{\cdot l\cdot j}dx^i\wedge dx^l.
\end{split}
\end{align*}
Thus
\begin{align*}
\begin{split}
&\quad\sigma_{2k}(\widehat{\Omega}) \\
&=\frac{1}{2^{2k}}
\sum_{i_1<i_2<\cdots<i_{2k},\ \sigma\in S_{2k}}
\operatorname{sgn}(\sigma)
\widehat R_{\cdot l_1\cdot i_1}\cdots \widehat R_{\cdot l_{2k}\cdot i_{2k}}
dx^{i_{\sigma(1)}}\wedge dx^{l_1}\wedge \cdots \wedge dx^{i_{\sigma(2k)}}\wedge dx^{l_{2k}}\\
&=\frac{1}{2^{2k}}
\sum_{i_1<i_2<\cdots<i_{2k},\ \sigma\in S_{2k}}
\operatorname{sgn}(\sigma)
\widehat R_{\cdot l_1\cdot i_1}\cdots \widehat R_{\cdot l_{2k}\cdot i_{2k}}
dx^{i_{\sigma(1)}}\wedge\cdots\wedge dx^{i_{\sigma(2k)}}
\wedge\cdots \wedge dx^{l_1}\wedge \cdots\wedge dx^{l_{2k}}\\
&=\frac{1}{2^{2k}}
\sum_{i_1<i_2<\cdots<i_{2k},\ \sigma\in S_{2k}}
\operatorname{sgn}(\sigma)
R_{\cdot l_1\cdot i_1}\cdots R_{\cdot l_{2k}\cdot i_{2k}}
dx^{i_{\sigma(1)}}\wedge\cdots\wedge dx^{i_{\sigma(2k)}}
\wedge\cdots\wedge dx^{l_1}\wedge\cdots\wedge dx^{l_{2k}}\\
&=\frac{1}{2^{2k}}
\sum_{i_1<i_2<\cdots<i_{2k}}
R_{\cdot l_1\cdot i_1}\cdots R_{\cdot l_{2k}\cdot i_{2k}}
dx^{i_1}\wedge\cdots\wedge dx^{i_{2k}}
\wedge\cdots\wedge dx^{l_1}\wedge\cdots\wedge dx^{l_{2k}}.
\end{split}
\end{align*}

Since $R_{\cdot i\cdot j}=R_{\cdot j\cdot i}$, all the relevant terms vanish identically. Hence
\[
\sigma_{2k}(\widehat\Omega)=0.
\]
\hfill {\it Q.E.D.}

\begin{example}[\cite{SZ}]
It is a known fact that the Pontrjagin classes of $\mathbb{C}P^n$ $(n\geq 3)$ are nonzero. Therefore, $\mathbb{C}P^n$ $(n\geq 3)$ does not admit a locally projectively flat spray.
\end{example}

\section{Sprays of Scalar Curvature and Isotropic Curvature}

A spray is said to have \emph{scalar curvature} if
\[
R^i_{\ k}=R\delta^i_k-\tau_ky^i,
\]
where $R=R(x,y)$, $\tau_k=\tau_k(x,y)$, and
\[
\tau_ky^k=R.
\]
A spray is said to have \emph{isotropic curvature} if
\[
R^i_{\ k}=R\delta^i_k-{1\over2}R_{\cdot k}y^i.
\]
Thus, in the isotropic case,
\[
\tau_k={1\over2}R_{\cdot k}.
\]
From Lemma~\ref{lemWeyl}, it is easy to obtain the following lemma.
\begin{lemma}[\cite{SL}]\label{chi_zero_scalar}
Let $G$ be a spray of scalar curvature. Then $G$ has isotropic curvature if and only if its $\chi$-curvature vanishes.
\end{lemma}

\begin{lemma}\label{chi_zero}
Suppose
\[
G^i=\overline G^i+Py^i,
\]
so that $G$ and $\overline G$ are projectively related. Then their $\chi$-curvatures satisfy
\[
\chi_i=\overline\chi_i+{n+1\over2}
(P_{\cdot i|m}y^m-P_{|i}),
\]
where `` $|$" denotes covariant differentiation with respect to $G$.
\end{lemma}

\noindent {\bf Proof.}
From $G^m = \overline{G}^m + P y^i$, a direct computation gives that the $S$-curvatures $S$ and $\bar{S}$ of $G$ and $\overline{G}$ satisfy
\[
S = \bar{S} + (n+1) P.
\]
From (\ref{chi}), we obtain
\begin{align*}
\begin{split}
 \chi_i &= \frac{1}{2}(S_{ \cdot i|m}y^m - S_{|i}) = \frac{1}{2} (\bar{S}_{\cdot i|m}y^m - \bar{S}_{|i}) + \frac{n+1}{2} (P_{\cdot i|m}y^m - P_{|i}) \\
        &= \bar{\chi}_i + \frac{n+1}{2} (P_{\cdot i|m}y^m - P_{|i}).
\end{split}
\end{align*}
\hfill {\it Q.E.D.}

\begin{theorem}\label{thmGGbar}
Suppose
\[
G^i=\overline G^i+Py^i,
\]
and suppose that $\overline G$ has isotropic curvature. Then $G$ has isotropic curvature if and only if
\[
P_{\cdot i|m}y^m-P_{|i}=0.
\]
In local coordinates this is
\begin{align}\label{P_local}
y^m{\partial^2P\over\partial x^m\partial y^i}
-{\partial P\over\partial x^i}
-2\overline G^m{\partial^2P\over\partial y^i\partial y^m}=0.
\end{align}
\end{theorem}

\noindent {\bf Proof.}
By Lemma~\ref{chi_zero_scalar}, the $\chi$-curvature of $\overline{G}$ vanishes, and $G$ has isotropic curvature if and only if the $\chi$-curvature of $G$ vanishes.
Then, by Lemma~\ref{chi_zero}, $G$ has isotropic curvature if and only if
\[
P_{\cdot i|m}y^m - P_{|i}=0.
\]
In local coordinates, the above equation becomes
\[
y^m \frac{\partial^2 P}{\partial x^m \partial y^i}
-\frac{\partial P}{\partial x^i}
-2 G^m\frac{\partial^2 P}{\partial y^i \partial y^m}=0.
\]
Substituting $G^m = \overline{G}^m + P y^m$ into the above equation, and using the fact that $P$ is positively homogeneous of degree one in $y$ (here we use $\frac{\partial^2 P}{\partial y^i \partial y^m} y^m =0$), we obtain
(\ref{P_local}).
\hfill {\it Q.E.D.}

The choice of $P$ and $\overline G$ in this theorem is an interesting problem. If both sprays are projectively flat, the condition is naturally satisfied whenever both have isotropic curvature. The next example illustrates how one may construct non-projectively-flat examples.

\begin{example}
Let $\overline G$ be the spray of a Randers metric of constant flag curvature, and let $P$ be a one-form satisfying (\ref{P_local}). For instance, one may take
\[
P=\langle a,y\rangle,
\]
where $a$ is a nonzero constant vector and $\langle\ ,\ \rangle$ is the Euclidean inner product. Then
\[
G^i=\overline G^i+Py^i
\]
has isotropic curvature and is projectively related to $\overline G$. By known results on Randers metrics, such examples can be non-projectively flat and need not be induced by a Finsler metric.
\end{example}

Every Finsler metric induces a spray, where here the Finsler metric may be pseudo-Finsler. For simplicity, in what follows we restrict attention to positive definite Finsler metrics.

\begin{lemma}
Let $G$ be a spray on an $n$-dimensional manifold $M$, and let $F$ be a Finsler metric on $M$. Then the spray coefficients $G^i_F$ of $F$ are
\begin{align}\label{G^i_F}
G^i_F=G^i+Py^i+Q^i,
\end{align}
where
\[
P={F_{|m}y^m\over2F},
\qquad
Q^i={F\over2}g^{il}(F_{\cdot l|m}y^m-F_{|l}).
\]
Here the covariant derivatives of $F$ are taken with respect to $G$.
\end{lemma}

\noindent {\bf Proof.}
Let $L=F^2$. A direct computation gives
\[
L_{\cdot l|m}y^m-L_{|l}
=L_{y^lx^m}y^m-L_{x^l}-4g_{lk}G^k.
\]
Hence
\[
G^i_F
={1\over4}g^{il}(L_{y^lx^m}y^m-L_{x^l})
=G^i+{1\over4}g^{il}(L_{\cdot l|m}y^m-L_{|l}).
\]
Replacing $L$ by $F^2$ yields (\ref{G^i_F}).
\hfill {\it Q.E.D.}

\begin{proposition}[\cite{Rap}]
Let $(M,G)$ be a spray manifold. A Finsler metric $F$ is projectively equivalent to $G$ if and only if
\[
F_{\cdot l|m}y^m-F_{|l}=0.
\]
\end{proposition}

Combining this proposition with the previous theorem gives the following.

\begin{theorem}
Let $(M,G)$ be a spray manifold of isotropic curvature and let $F$ be a Finsler metric on $M$. Define
\[
\widetilde G^i=G^i+Fy^i.
\]
Then $\widetilde G$ has isotropic curvature if and only if $F$ is projectively equivalent to $G$.
\end{theorem}

\section{The $\chi$-Curvature of a Finsler Metric}

Let $F$ be a Finsler metric on an $n$-dimensional manifold $M$. Its spray $G_F$ has Riemann curvature $R^{\ i}_{j\ kl}$ and Berwald curvature $B^{\ i}_{j\ kl}$. Let
\[
g_{ij}={1\over2}[F^2]_{y^iy^j}.
\]
Two basic non-Riemannian quantities are the Cartan torsion
\[
C_{ijk}={1\over2}{\partial g_{ij}\over\partial y^k}
\]
and the Landsberg curvature
\begin{align}\label{L_ijk}
L_{ijk}=-{1\over2}y^sg_{sm}B^{\ m}_{i\ jk}
=C_{ijk|m}y^m.
\end{align}
Define the mean Cartan torsion and mean Landsberg curvature by
\[
I_k=g^{ij}C_{ijk},
\qquad
J_k=g^{ij}L_{ijk}.
\]
These quantities are not independent. One has the Ricci identity
\[
L_{ijk|l}-L_{ijl|k}+C_{ijp}R^{\ p}_{\ kl}
=-{1\over2}\{g_{pj}R^{\ p}_{i\ kl}+g_{ip}R^{\ p}_{j\ kl}\}.
\]
Contracting with $y^l$ gives
\begin{align}\label{LCTheta}
L_{ijk|l}y^l+C_{ijp}R^{\ p}_{\ k}
=-{1\over2}\{g_{pj}R^{\ p}_{i\ kl}+g_{ip}R^{\ p}_{j\ kl}\}y^l.
\end{align}
The left-hand side is a non-Riemannian curvature quantity which vanishes for Riemannian metrics. The right-hand side also shows that it vanishes when $F$ has isotropic curvature.

Contracting (\ref{LCTheta}) with $g^{ij}$ yields
\begin{align}\label{JI}
J_{k|l}y^l+I_p R^{\ p}_{\ k}=-R^{\ m}_{m\ kl}y^l.
\end{align}
Combining (\ref{JI}) with Lemma \ref{chi_k_R} gives the following.

\begin{lemma}[\cite{Mo,MoSh}]
\[
\chi_k={1\over2}\{J_{k|l}y^l+I_p R^{\ p}_{\ k}\}.
\]
\end{lemma}

From (\ref{L_ijk}),
\[
J_{k|l}y^l=(I_{k|m}y^m)_{| l}y^l
=I_{k|m|l}y^my^l.
\]
Hence
\begin{align}\label{2chi}
I_{k|m|l}y^my^l+I_p R^{\ p}_{\ k}=2\chi_k.
\end{align}

\begin{theorem}[\cite{Sh2}]
Let $(M,F)$ be a complete Finsler manifold with $\chi_k=0$. Assume that the flag curvature is negative and the mean Cartan torsion is bounded. Then $F$ is a Riemannian metric.
\end{theorem}

\noindent {\bf Proof.}
By the assumption, (\ref{2chi}) reduces to
\begin{align}\label{chi_k=0}
I_{k|m|l} y^m y^l + I_{p}{R}^{\ p}_{\ k} = 0.
\end{align}
Let $c=c(t)$ be any geodesic defined on $(-\infty,\infty)$. Let $v=v(t)$ be a parallel vector field along $c$.
Then $I(t):=I_{c'(t)}(v(t))$ is a vector field along $c$. In this case, (\ref{chi_k=0}) can be written as
\[
D_{c'}D_{c'}I(t)+R_{c'}(I(t))=0.
\]
Set
\[
\varphi(t)=g_{c'(t)}(I(t),I(t)).
\]
Then
\begin{align*}
\begin{split}
\varphi''(t)
&=2g_{c'}(D_{c'}D_{c'}I(t),I(t))+2g_{c'}(D_{c'}I(t),D_{c'}I(t))\\
&=-2g_{c'}(R_{c'}(I(t)),I(t))+2g_{c'}(D_{c'}I(t),D_{c'}I(t))\\
&\geq 0.
\end{split}
\end{align*}
Therefore, $\varphi(t)$ is a nonnegative convex function. Suppose at some point $t_0$, $\varphi'(t_0)\neq 0$. When $\varphi'(t_0)<0$, we have
\[
\varphi(t)\geq \varphi(t_0)-\varphi'(t_0)(t_0-t),\qquad t<t_0.
\]
When $\varphi'(t_0)>0$, we have
\[
\varphi(t)\geq \varphi(t_0)+\varphi'(t_0)(t-t_0),\qquad t>t_0.
\]
It follows that in either case $\varphi(t)$ is unbounded. This shows that for all $t$, $\varphi'(t)=0$ and $\varphi''(t)=0$.
Thus
\[
R_{c'}(I(t))=0,\qquad D_{c'}I(t)=0.
\]
In particular, $R_g(I_g(t))=0$. Hence $I_g(v)=0$, i.e., the mean Cartan torsion vanishes. Then by Deicke's theorem, $F$ is a Riemannian metric.
\hfill {\it Q.E.D.}

\begin{corollary}[\cite{AZ}]
Let $(M,F)$ be a complete Finsler manifold of constant flag curvature $K<0$. If the mean Cartan torsion is bounded, then $F$ is Riemannian.
\end{corollary}

\noindent {\bf Proof.}
A Finsler metric of constant flag curvature has vanishing $\chi$-curvature. The result follows from the preceding theorem.
\hfill {\it Q.E.D.}

\section{The Inverse Problem}

A Finsler metric on a manifold induces a spray, but a spray need not be induced by a metric. This leads to the inverse problem: how can one determine whether a given spray is induced by a Finsler metric?

A complete answer is not known in general. Bucataru and Muzsnay gave, for sprays with nonzero Ricci curvature, equivalent conditions for metrizability by Finsler metrics of nonzero scalar flag curvature, and they also investigated the locally projectively flat case \cite{Bucatarul,MZ1}. Yang constructed classes of locally projectively flat sprays that are not metrizable \cite{Y}. A different approach, based on the classification of locally projectively flat Finsler metrics of constant flag curvature \cite{Li1,Sh3}, provides direct conditions in terms of the projective factor $P$, including the zero-curvature case.

\begin{proposition}[\cite{SL}]\label{propPPP}
Let
\[
G=y^i{\partial\over\partial x^i}
-2P(x,y)y^i{\partial\over\partial y^i}
\]
be a locally projectively flat spray on an $n$-dimensional manifold $M$. Then $G$ has isotropic curvature if and only if
\begin{align}\label{Pisocurv}
y^m{\partial^2P\over\partial x^m\partial y^i}
-{\partial P\over\partial x^i}=0.
\end{align}
\end{proposition}

\noindent {\bf Proof.}
Since $G$ is projectively related to the zero spray, this follows from Theorem \ref{thmGGbar} by setting $\overline G^m=0$ in (\ref{P_local}).
\hfill {\it Q.E.D.}

If a spray of isotropic curvature is induced by a Finsler metric, then the corresponding Finsler metric has isotropic flag curvature. For locally projectively flat Finsler metrics, a classical theorem of Berwald implies that isotropic flag curvature is necessarily constant.

\begin{lemma}[\cite{Ber2}]\label{KisoKconst}
Let $F$ be a locally projectively flat Finsler metric with isotropic flag curvature $K=K(x)$. Then $K$ is constant.
\end{lemma}

Berwald also established the following characterization.

\begin{lemma}[\cite{Ber2}]\label{lemBerwald}
Let $F$ be a Finsler metric on an $n$-dimensional manifold $M$. Then $F$ is locally projectively flat if and only if there exist a positively homogeneous function $P=P(x,y)$ of degree one and a positively homogeneous function $K=K(x,y)$ of degree zero such that
\begin{align}\label{Berwald1}
F_{x^k}=(PF)_{y^k},
\end{align}
\begin{align}\label{Berwald2}
P_{x^k}-PP_{y^k}
=-{1\over3F}(KF^3)_{y^k}.
\end{align}
In this case $P$ is the projective factor of $F$ and $K$ is its flag curvature.
\end{lemma}

\begin{theorem}[\cite{SL}]\label{inverse1}
Let
\[
G=y^i{\partial\over\partial x^i}
-2P(x,y)y^i{\partial\over\partial y^i}
\]
be a locally projectively flat spray on an $n$-dimensional manifold $M$. Then $G$ has isotropic curvature if and only if
\begin{align}\label{PP}
P_{x^k}=P_{y^kx^m}y^m.
\end{align}
In this case, if
\[
P^2-P_{x^m}y^m\ne0
\]
is neither the square of a Finsler metric nor the negative of such a square, then $G$ cannot be induced by any Finsler metric.
\end{theorem}

\noindent {\bf Proof.}
By Proposition \ref{propPPP}, $G$ has isotropic curvature if and only if $P$ satisfies (\ref{Pisocurv}). If $G$ were induced by a Finsler metric $F$, then (\ref{Berwald2}) would imply
\begin{align}\label{P2KF}
P^2-P_{x^m}y^m=K(x)F^2.
\end{align}
When the left-hand side is nonzero, Lemma \ref{KisoKconst} implies that $K(x)$ is a nonzero constant. Hence $P^2-P_{x^m}y^m$ must be a nonzero constant multiple of the square of a Finsler metric, contrary to the assumption.
\hfill {\it Q.E.D.}

For locally projectively flat sprays of zero Riemann curvature, the classification of projectively flat Finsler metrics of zero flag curvature gives a positive metrizability result.

\begin{proposition}[\cite{SL}]\label{propFFF}
Let
\[
G=y^i{\partial\over\partial x^i}
-2P(x,y)y^i{\partial\over\partial y^i}
\]
be a locally projectively flat spray. Then $G$ has zero curvature if and only if
\begin{align}\label{FunkF}
P_{x^k}=PP_{y^k}.
\end{align}
\end{proposition}

\begin{theorem}[\cite{SL}]\label{thm1.3}
Let
\[
G=y^i{\partial\over\partial x^i}
-2P(x,y)y^i{\partial\over\partial y^i}
\]
be a locally projectively flat spray. If $G$ has zero curvature, then it is induced by a family of Finsler metrics.
\end{theorem}

\begin{example}[\cite{SL}]\label{ex4.1}
Suppose $P=P(x,y)$ satisfies (\ref{FunkF}). Then
\[
F=\psi(y+xP)\{1+P_{y^k}x^k\}
\]
is a Finsler metric, not necessarily positive definite, where $\psi=\psi(y)$ is a Minkowski norm, again not necessarily positive definite.

As a special case, let
\[
P={\sqrt{(1-|x|^2)|y|^2+\langle x,y\rangle^2}+\langle x,y\rangle
\over 1-|x|^2},
\qquad
\psi=|y|+\langle a,y\rangle,
\]
where $a$ is a constant vector. Then one obtains a locally projectively flat Finsler metric of zero curvature:
\begin{align}\label{Berwaldmetric}
\begin{split}
F=&\Big\{1+\Big\langle a,\,x+\frac{(1-|x|^2)\langle a,y\rangle}
{\sqrt{(1-|x|^2)|y|^2+\langle x,y\rangle^2}+\langle x,y\rangle}\Big\rangle\Big\}\\
&\times\frac{\big(\sqrt{(1-|x|^2)|y|^2+\langle x,y\rangle^2}+\langle x,y\rangle\big)^2}
{(1-|x|^2)^2\sqrt{(1-|x|^2)|y|^2+\langle x,y\rangle^2}}.
\end{split}
\end{align} 
When $a=0$, this reduces to the classical Berwald metric.
\end{example}

Weakly isotropic sprays and additional metrizable and nonmetrizable examples were studied further in \cite{LL}. Other constructions of metrizable sprays were given in \cite{LMY}. The general inverse problem for sprays without special curvature assumptions remains largely open, and an effective general criterion is not yet known.

\end{document}